\documentclass[a4paper, 12pt, reqno]{amsart}
\usepackage{amsmath}
\usepackage{amssymb}
\usepackage{graphics}
\usepackage{mathrsfs}
\usepackage[pdflatex]{graphicx}
\usepackage{hyperref}
\usepackage[marginratio=1:1,tmargin=117pt, height=650pt]{geometry}
\usepackage{color}
\usepackage{amsthm}

\usepackage[all]{xy}

\usepackage{eurosym}

\usepackage[english]{babel}
\usepackage[T1]{fontenc}
\usepackage[latin1]{inputenc}
\usepackage{marginnote}

\newcommand{\tabitem}{~~\llap{\textbullet}~~}

\usepackage{amsmath,amsfonts,amsthm,amssymb,verbatim,enumerate,quotes,graphicx}
\usepackage[flushmargin]{footmisc}
\usepackage[noadjust]{cite}
\usepackage{color}
\newcommand{\comments}[1]{}
\numberwithin{equation}{section}
\makeatletter
\def\blfootnote{\xdef\@thefnmark{}\@footnotetext}
\makeatother

\definecolor{orange}{rgb}{1,0.5,0}

\newcommand{\ds}{\displaystyle}
\newcommand{\interior}{\mbox{int}\,}

\theoremstyle{plain}
\newtheorem{thm}{Theorem}[section]

\newtheorem{lem}[thm]{Lemma}

\theoremstyle{definition}
\newtheorem{dfn}{Definition}[section]

\theoremstyle{remark} 

\newtheorem{rmk}{Remark}[section]

\newcommand{\C}{{\mathbb{C}}}
\newcommand{\CR}{{\hat{\mathbb{C}}}}
\newcommand{\CS}{{\mathbb{C}^*}}
\newcommand{\B}{\mathcal B}

\newcommand{\Z}{{\mathbb{Z}}}
\newcommand{\N}{{\mathbb{N}}}

\begin{document}

\bibliographystyle{amsalpha}

\title[The escaping set of transcendental self-maps of $\C^*$]{The escaping set of transcendental self-maps of the punctured plane}
\author[D. Mart\'i-Pete]{David Mart\'i-Pete}
\address{Department of Mathematics and Statistics\\ The Open University\\ Walton Hall\\ Milton Keynes MK7 6AA\\ United Kingdom}
\email{david.martipete@open.ac.uk}
\date{\today}

\maketitle

\begin{abstract}
We study the different rates of escape of points under iteration by holomorphic self-maps of $\mathbb C^*=\mathbb C\setminus \{0\}$ for which both $0$ and $\infty$ are essential singularities. Using annular covering lemmas we construct different types of orbits, including fast escaping and arbitrarily slowly escaping orbits to either 0, $\infty$ or both. We also prove several properties about the set of fast escaping points for this class of functions. In particular, we show that there is an uncountable collection of disjoint sets of fast escaping points each of which has $J(f)$ as its boundary.
\end{abstract}

\section{Introduction}

Complex dynamics concerns the iteration of a holomorphic function on a Riemann surface $S$. Given a point $z\in S$, we consider the sequence given by its iterates $f^n(z)=(f\circ\ds\mathop{\cdots}^{n}\circ f)(z)$ and study the possible behaviours as $n$ tends to infinity. We partition $S$ into the \textit{Fatou set}, or set of stable points,
$$
F(f):=\bigl\{z\in S\ :\ (f^n)_{n\in\mathbb N} \mbox{ is a normal family in some neighbourhood of } z\bigr\}
$$
and the \textit{Julia set} $J(f):=S\setminus  F(f)$, consisting of chaotic points. If $f:S\rightarrow S$ is holomorphic and $\CR\setminus S$ consists of essential singularities, then there are three interesting cases:
\begin{itemize}
\item $S=\CR:=\C\cup\{\infty\}$ and $f$ is a rational map;
\item $S=\C$ and $f$ is a transcendental entire function;
\item $S=\CS:=\C\setminus\{0\}$ and \textit{both} $0$ and $\infty$ are essential singularities.
\end{itemize}
We study this third class of maps, which we call \textit{transcendental self-maps of} $\CS$. Note that if $f$ has three or more omitted points, then Montel's theorem tells us that $f$ is constant. In particular, $f:\C^*\rightarrow \C^*$ has no omitted values. A basic reference on iteration theory in one complex variable is \cite{milnor06}. See \cite{bergweiler93} for a survey on transcendental entire and meromorphic functions. 

The iteration of transcendental (entire) functions dates back to the times of Fatou \cite{fatou26}. However, we have to wait until 1953 to find a paper about the iteration of holomorphic self-maps of $\C^*$ \cite{radstrom53}. Bhattacharyya in his PhD thesis \cite{bhattacharyya} under the supervision of Prof. Noel Baker showed that such maps are all of the form
\begin{equation}
f(z)=z^n\exp\bigl(g(z)+h(1/z)\bigr),
\label{eqn:bhat}
\end{equation}
where $n\in\mathbb Z$ and $g,h$ are non-constant entire functions. Since then many other people have studied them; for example, see \cite{baker87}, \cite{kotus87}, \cite{makienko87}, \cite{keen86}, \cite{liping91}, \cite{mukhamedshin91},  \cite{hinkkanen94}, \cite{bergweiler95} and \cite{baker-dominguez98}. Such maps arise in a natural way when you complexify circle maps. For instance, the so-called Arnol'd standard family of perturbations of rigid rotations 
$$
f_{\alpha\beta}(\theta)=\theta+\alpha+\beta\sin(\theta)\pmod{2\pi},\quad \mbox{ where }0\leqslant \alpha\leqslant 2\pi,\ \beta\geqslant 0,
$$
has as its complexification $\widehat{f}_{\alpha\beta}(z)=ze^{i\alpha}e^{\beta(z-1/z)/2}$ \cite{fagella99}. 

\nocite{liping93}
\nocite{liping97}
\nocite{liping98}
\nocite{keen89}
\nocite{kotus90}
\nocite{makienko91}

\nocite{fatou26}
\nocite{fatou19}
\nocite{julia18}


In this paper we study the escaping set of transcendental self-maps of $\C^*$. For an \textit{entire} function $f$, the \textit{escaping set} of $f$ is defined by
$$
I(f):=\{z\in\C :\ f^n(z)\rightarrow \infty \mbox{ as } n\rightarrow \infty\}.
$$ 
If $f$ is a polynomial, then infinity is an attracting fixed point and the escaping set consists of the basin of attraction of infinity, $\mathcal A(\infty)$, which is part of the Fatou set and is connected. Moreover, $J(f)=\partial \mathcal A(\infty)$, so the escaping set has a close relationship to the Julia set.

For transcendental entire functions, the escaping set also plays an important role. It was first studied by Eremenko \cite{eremenko89} who used Wiman-Valiron theory to show that, for a transcendental entire function $f$, 
\begin{enumerate}
\item[I1)] $I(f)\cap J(f) \neq \emptyset$;
\item[I2)] $J(f) = \partial I(f)$;
\item[I3)] the components of $\overline{I(f)}$ are unbounded.
\end{enumerate} 
Furthermore, in \cite{eremenko-lyubich92} the so-called \textit{Eremenko-Lyubich class}
$$
\mathcal B:=\{f \mbox{ trancendental entire}\ :\ \mbox{sing}(f^{-1}) \mbox{ is bounded}\}
$$
was introduced (the set $\mbox{sing}(f^{-1})$ consists of the critical values and asymptotic values of $f$) and the authors showed that 
\begin{enumerate}
\item[I4)] if $f\in\mathcal B$, then $I(f)\subseteq J(f)$. 
\end{enumerate}
In his article \cite{eremenko89} Eremenko conjectured that all the components of the escaping set are also unbounded. In 2011 it was proved that this is true for functions of finite order in class $\B$ \cite{rrrs11}, but for general transcendental entire functions this remains an open question.

Key progress on Eremenko's conjecture was obtained by studying the \textit{fast escaping set} defined by
$$
A(f):=\{z\in\C\ :\ \exists \ell\in\N,\ |f^{n+\ell}(z)|\geqslant M^n(R,f) \mbox{ for all } n\in\mathbb \N\},
$$
where $M(R,f)=\max_{|z|=R}|f(z)|$ and $R>0$ is chosen to be sufficiently large so that $M^n(R,f)\rightarrow+\infty$ as $n\rightarrow \infty$. Here and throughout $\N=\{0,1,2,\hdots\}$.

The set $A(f)$, which consists of the points that escape about as fast as possible, was introduced by Bergweiler and Hinkkanen, and it shares some properties with $I(f)$, for example, $J(f)=\partial A(f)$ and $J(f)\cap A(f)\neq\emptyset$, see \cite{bergweiler-hinkkanen99} and \cite{rippon-stallard05}. But it also has some much nicer properties. Rippon and Stallard showed that all the components of $A(f)$ are unbounded, and hence $I(f)$ has at least one unbounded component. The paper \cite{rippon-stallard12} gives a compilation of results about $A(f)$.\pagebreak

For transcendental self-maps of $\C^*$ we define the escaping set by
$$
I(f):=\left\{z\in\C^*\ :\ \omega(z,f)\subseteq \{0,\infty\} \right\}, 
$$
where $\omega(z,f):=\bigcap_{n\in\N} \overline{\{f^k(z)\ :\ k\geqslant n\}}$. The set $I(f)$ contains points that escape to $0$ as well as to $\infty$, 
$$
\begin{array}{c}
I_0(f):=\left\{z\in\C^*\ :\ f^n(z)\rightarrow 0\mbox{ as } n\rightarrow \infty\right\},\vspace{10pt}\\
I_\infty(f):=\left\{z\in\C^*\ :\ f^n(z)\rightarrow \infty \mbox{ as } n\rightarrow \infty\right\},
\end{array}
$$
and also contains points that escape from $\C^*$ by jumping infinitely many times between a neighbourhood of $0$ and a neighbourhood of $\infty$. The sets $I_0(f)$ and $I_\infty(f)$ were studied by Fang \cite{liping98} who proved that 
$$I_0(f)\cap J(f)\neq \emptyset,\quad I_\infty(f)\cap J(f)\neq\emptyset\quad \mbox{ and }\quad J(f)=\partial I_0(f)=\partial I_\infty(f),
$$
by using Wiman-Valiron theory in the way that Eremenko did for the entire case, but the full set of escaping points $I(f)$ has not previously been studied.

To classify the various types of escaping orbits we introduce the following\linebreak concept.

\begin{dfn}[Essential itinerary]
Let $f$ be a transcendental self-map of $\C^*$. We define the \textit{essential itinerary} of a point $z\in I(f)$ to be the symbol sequence $e=(e_n)\in \{0,\infty\}^\mathbb N$ such that
$$
e_n:=\left\{
\begin{array}{ll}
0, & \mbox{ if } |f^n(z)|\leqslant 1,\vspace{10pt}\\
\infty, & \mbox{ if } |f^n(z)|> 1,
\end{array}
\right.
$$
for all $n\in\N$.
\label{dfn:essential-itinerary}
\end{dfn}

For each $e\in \{0,\infty\}^\mathbb N$, the set of escaping points whose essential itinerary is eventually a shift of $e$ is
$$
I_e(f):=\{z\in I(f)\ :\ \exists\ell,k\in\N,\ \forall n\geqslant 0,\ |f^{n+\ell}(z)|>1~\Leftrightarrow~ e_{n+k}=\infty\}.
$$

We also have a notion of \textit{fast escaping} points related to an essential itinerary $e$, defined using the iterates of the maximum and minimum modulus functions
$$
M(r,f):=\max_{|z|=r}|f(z)|<+\infty,\quad m(r,f):=\min_{|z|=r}|f(z)|>0.
$$

\begin{dfn}[Fast escaping set]
Let $f$ be a transcendental self-map of $\CS$. Let $e=(e_n)\in\{0,\infty\}^\N$ and let $R_0>0$ be sufficiently large (or small) so that the sequence $(R_n)$ defined for $n>0$ by 
\begin{itemize}
\item $R_{n+1}=m(R_n)$, if $e_{n+1}=0$,
\item $R_{n+1}=M(R_n)$, if $e_{n+1}=\infty$,
\end{itemize}
accumulates to $\{0,\infty\}$. We say that a point $z\in\C^*$ is \emph{fast escaping} if there are $\ell,k\in\N$ such that 
\begin{itemize}
\item $|f^{n+\ell}(z)| \leqslant R_n$, if $e_{n+k}=0$,
\item $|f^{n+\ell}(z)| \geqslant R_n$, if $e_{n+k}=\infty$,
\end{itemize}
for all $n\in\N$. We denote the set of all fast escaping points by $A(f)$ and the set of fast escaping points with essential itinerary $e\in \{0,\infty\}^\N$ by $A_e(f)$.
\label{dfn:fast-escaping}
\end{dfn}

Observe that if $f$ is of the form \eqref{eqn:bhat}, then the behaviour of $f$ in a neighbourhood of $\infty$ depends mainly on that of the entire function $g$ while the behaviour near $0$ depends mainly on that of $h$.

We begin by proving an analogue of property (I1), namely that $I_e(f)$ and indeed $A_e(f)$ are non-empty for \textit{any} essential itinerary $e$. We follow the approach of Rippon and Stallard in \cite{rippon-stallard13} where they proved the existence of points escaping to infinity at different rates by constructing points with different annular itineraries.

\begin{thm}
Let $f$ be a transcendental self-map of $\C^*$. For each $e\in \{0,\infty\}^\mathbb N$, $A_e(f)\cap J(f)\neq \emptyset$ and hence $I_e(f)\cap J(f)\neq \emptyset$.
\label{thm:fast-escaping-set-not-empty}
\end{thm}


Our notation for annular itineraries is as follows. Let $R_+>0$ and $R_->0$ be respectively large enough and small enough such that, for all $r>R_+$, $M(r)>r$ and, for all $0<r<R_-$, we have $m(r)<r$. Then define
$$
A_0:=\{z\in\C^*\ :\ R_-< |z|< R_+\}
$$
and the following sequences of annuli:
$$
\begin{array}{ll}
A_n:=\{z\in\C^*\ :\ M^{n-1}(R_+)\leqslant |z|< M^n(R_+)\}, & \mbox{for } n>0;\vspace{10pt}\\
A_n:=\{z\in\C^*\ :\ m^{-n}(R_-)<|z|\leqslant m^{-n+1}(R_-)\},  & \mbox{for } n<0.
\end{array}
$$
Each point $z\in I(f)$ has an associated \textit{annular itinerary} $(s_n)\in \Z^\N$ with respect to the partition $\{A_n\}$ such that $f^n(z)\in A_{s_n}$ for all $n\in\N$. We prove a covering result (see Theorem \ref{thm:annular-itineraries}) which allows us to construct orbits with certain annular itineraries including the ones listed in Theorem \ref{thm:types-annular-itineraries} below.

\begin{rmk}
In this article we deal with two kinds of itineraries for escaping points that should not be confused: essential itineraries $(e_n)\in \{0,\infty\}^\N$ that describe how an escaping point accumulates to the two essential singularities, and annular itineraries $(s_n)\in \Z^\N$ that depend on the partition $\{A_n\}$. For large values of $n$, $e_n=0$ and $e_n=\infty$ correspond respectively to negative and positive terms in the annular itinerary.
\end{rmk}

\begin{thm}
Let $f$ be a transcendental self-map of $\C^*$. Given an annular partition $\{A_n\}$  defined as above with $R_+, 1/R_-$ sufficiently large, we can construct points with the following itineraries:
\begin{itemize}
\item fast escaping itineraries;
\item periodic itineraries;
\item bounded itineraries (uncountably many);
\item unbounded non-escaping itineraries (uncountably many);
\item arbitrarily slowly escaping itineraries.
\end{itemize}
\label{thm:types-annular-itineraries}
\end{thm}

Note that our proof uses a different annular covering lemma to those used in \cite{rippon-stallard13} and, in this setting we are able to avoid the exceptional sets which feature in \cite[Theorem 1.1 and Theorem 1.2]{rippon-stallard13}.

We now state a result in the spirit of property (I2) but for any essential itinerary~$e$. For the special cases of $I_0(f)$ and $I_\infty(f)$ this is due to Fang and it also follows from the results in \cite{baker-dominguez-herring01}.

\begin{thm}
Let $f$ be a transcendental self-map of $\C^*$. For each $e\in \{0,\infty\}^\mathbb N$, $J(f)=\partial A_e(f)=\partial I_e(f)$. Also $J(f)=\partial A(f)=\partial I(f)$.
\label{thm:boundaries}
\end{thm}

Since there are uncountably many different essential itineraries, in particular this means that there is an uncountable collection of disjoint sets, each of which has the Julia set as its boundary.

We can also prove the analogue of property (I3) for any essential itinerary. When we say that a set $X$ is \textit{unbounded} in $\C^*$ we mean that $\overline{X}\cap \{0,\infty\}\neq \emptyset$.

\begin{thm}
Let $f$ be a transcendental self-map of $\C^*$. For each $e\in \{0,\infty\}^\mathbb N$, the connected components of $\overline{I_e(f)}$ are unbounded, and hence the connected components of $\overline{I(f)}$ are unbounded.
\label{thm:components-closure-If-unbdd}
\end{thm}

Finally we show that, as for transcendental entire functions, the components of $A(f)$ are all unbounded.

\begin{thm}
Let $f$ be a transcendental self-map of $\C^*$. For each $e\in \{0,\infty\}^\mathbb N$, the connected components of $A_e(f)$ are unbounded, and hence the connected components of $A(f)$ are unbounded.
\label{thm:cc-fast-escaping-set-unbounded}
\end{thm}


In \cite{fagella-martipete} we study the escaping set of holomorphic self-maps of $\C^*$ of bounded type, 
$$
\mathcal B^*:=\{f:\C^*\rightarrow\C^*  \mbox{ trancendental}\ :\ \mbox{sing}(f^{-1}) \mbox{ is bounded away from } 0 \mbox{ and } \infty\}.
$$
In that paper, we prove the analogue of property (I4): if $f\in\B^*$ then $J(f)\subseteq I(f)$. Following the ideas of \cite{rrrs11}, we show that for functions in class $\B^*$ of finite order every escaping point can be connected to either $0$ or $\infty$ by a curve of points escaping uniformly.


\vspace{10pt}

\noindent 
\textbf{Structure of the paper.} In Section \ref{sec:properties-M-m} we prove the basic properties of $M(r)$ and $m(r)$ that we are going to need later. The discussion about the notions of essential itinerary and the fast escaping set is in Section \ref{sec:escaping-set}. Section \ref{sec:annular-itineraries} is devoted to the construction of the annular itineraries and the proof of Theorem \ref{thm:types-annular-itineraries}. The main result in this section, Theorem \ref{thm:annular-itineraries}, in fact allows you to construct many more types of orbits than the ones listed in the statement of Theorem \ref{thm:types-annular-itineraries}. Theorem \ref{thm:fast-escaping-set-not-empty} is proved in Section \ref{sec:fast-escaping-set} and we prove Theorems \ref{thm:boundaries}, \ref{thm:components-closure-If-unbdd} and \ref{thm:cc-fast-escaping-set-unbounded} in Section \ref{sec:the-rest}. In doing so we also show that if a Fatou component $U$ intersects the fast escaping set $A(f)$ then $\overline{U}\subseteq A(f)$ (see Theorem \ref{thm:fast-escaping-Fatou-components}). 

\vspace{10pt}

\noindent
\textbf{Notation.} In this paper $\N=\{0,1,2,\hdots\}$, $D(z_0,r)$ denotes the open disc of radius $r$ centered at $z_0$, and if $r_1,r_2>0$,
$$
A(r_1,r_2):=\{z\in \mathbb C\ :\ r_1<|z|<r_2\}  ~\mbox{ and }~ \overline{A}(r_1,r_2):=\{z\in \mathbb C\ :\ r_1\leqslant |z|\leqslant r_2\}.
$$

\vspace{10pt}

\noindent
\textbf{Acknowledgments.} The author would like to thank his supervisors Phil Rippon and Gwyneth Stallard for all their support and patient guidance in the preparation of this article, as well as N\'uria Fagella for useful discussions and suggesting such an interesting topic for my Ph.D. thesis.

\section{Properties of $M(r)$ and $m(r)$}

\label{sec:properties-M-m}

Before proving the annular covering results we need some basic properties of the maximum and minimum modulus functions. Note that we will not usually make explicit the dependence on $f$ and we will just write $M(r)$ and $m(r)$. As a consequence of the maximum modulus principle, both $M(r)$ and $m(r)$ are unimodal functions. In the following lemma we summarise their main properties. Throughout this section we will only prove the statements for $M(r)$ when $r\rightarrow +\infty$, and the other three statements can be deduced from these by using the facts that if $\tilde{f}(z)=f(1/z)$ then
$$
M(r,f)=M(1/r,\tilde{f})=\frac{1}{m(r,1/f)}=\frac{1}{m(1/r, 1/\tilde{f})}.
$$

\begin{lem}
Let $f$ be a transcendental self-map of $\CS$. The functions $M(r)$ and $m(r)$ satisfy the following properties:
\begin{enumerate}
\item[i)] $\dfrac{\log M(r)}{\log r}\rightarrow +\infty$, $\dfrac{\log m(r)}{\log r}\rightarrow -\infty$ as $r\rightarrow +\infty$, and\vspace{10pt}\\
$\dfrac{\log M(r)}{\log r}\rightarrow -\infty$, $\dfrac{\log m(r)}{\log r}\rightarrow +\infty$ as $r\rightarrow 0$; \vspace{5pt} 
\item[ii)] $\log M(r)$ and $-\log m(r)$ are convex functions of $\log r$; \vspace{5pt}
\item[iii)]  $\exists R=R(f)>0$ such that 
$$
M(r^k)\geqslant M(r)^k,\ m(r^k)\leqslant m(r)^k \mbox{ for every } r\geqslant R,\ k>1,
$$
and $\exists R'=R'(f)>0$ such that 
$$
M(r^k)\geqslant M(r)^k,\ m(r^k)\leqslant m(r)^k \mbox{ for every } r\leqslant R',\ k>1;
$$
\item[iv)] for $k>1,
\begin{array}[t]{l}
\dfrac{M(kr)}{M(r)}\rightarrow +\infty,\ \dfrac{m(kr)}{m(r)}\rightarrow 0 \mbox{ as } r\rightarrow +\infty, \mbox{ and }\vspace{10pt}\\ 
\dfrac{M(kr)}{M(r)}\rightarrow 0,\ \dfrac{m(kr)}{m(r)}\rightarrow +\infty \mbox{ as } r\rightarrow 0.
\end{array}$
\end{enumerate}
\label{lem:mrf}
\end{lem}
\begin{proof}
\begin{enumerate}
\item[i)] This property follows from the analogous result for transcendental entire functions using the fact that $f(z)=z^n\exp\bigl(g(z)+h(1/z)\bigr)$, where $n\in\Z$ and $g,h$ are non-constant entire functions (see \eqref{eqn:bhat}), so
$$
\lim_{r\rightarrow +\infty} \frac{\log M(r,f)}{\log r}=\lim_{r\rightarrow +\infty} \frac{\log M(r,\tilde{g})}{\log r} = +\infty,
$$
where $\tilde{g}(z)=z^n\exp g(z)$. 


\item[ii)] This means that $\phi(t)=\log M(\exp t)$ is a convex function of $t$ and the property is usually referred to as the Hadamard three circles theorem, see \cite{ahlfors53}. Observe that in the hypothesis of that theorem you only need that the function is analytic in an annulus $r_1<|z|<r_2$ and it therefore applies to holomorphic self-maps of $\C^*$.

\item[iii)] See \cite[Lemma 2.2]{rippon-stallard09} or \cite[Theorem 2.2]{bergweiler-rippon-stallard13} for the analogous result for transcendental entire functions. We reproduce the proof here for completeness.

Let $\phi(t)=\log M(\exp t)$. By property (i), $\phi(t)/t\rightarrow +\infty$ as $t\rightarrow +\infty$, so we can take $t_1\geqslant t_0>0$ large enough that
$$
\phi(t_0)>0 \quad \mbox{ and }\quad \frac{\phi(t)}{t}\geqslant \frac{\phi(t_0)}{t_0}, \mbox{ for } t\geqslant t_1.
$$
Let $\phi'$ denote the right derivative of $\phi$. Then, by property (ii) and the previous inequality,
$$
\phi'(t)\geqslant \frac{\phi(t)-\phi(t_0)}{t-t_0}\geqslant \frac{\varphi(t)}{t}, \mbox{ for } t\geqslant t_1.
$$
Hence $\varphi(t)/t$ is an increasing function for $t\geqslant t_1$. Thus, if $k>1$, then
$$
\frac{\phi(kt)}{kt}\geqslant \frac{\phi(t)}{t}, \mbox{ that is, } \phi(kt)\geqslant k\phi(t),
$$
for $t\geqslant t_1$. Taking exponentials on both sides we get the result, with $R=\exp t_1$.

\item[iv)] For every value of $r>1$ we can write $kr=r^c$, where
$$
c=c(r)=\frac{\log k+\log r}{\log r}>1.
$$
By property (iii), for $r$ large enough,
$$
\frac{M(kr)}{M(r)}=\frac{M(r^c)}{M(r)}\geqslant \frac{M(r)^c}{M(r)}=M(r)^{c-1}
$$
and then, using property (i),
\begin{center}
$
\log \left(M(r)^{c-1}\right)=(c-1)\log M(r)=\dfrac{\log k}{\log r}\log M(r)\rightarrow +\infty\quad \mbox{ as } r\rightarrow +\infty,  
$
\end{center}
so 
$$
\frac{M(kr)}{M(r)}\rightarrow +\infty\quad \mbox{ as } r\rightarrow +\infty.
$$
\end{enumerate}
\end{proof}

The following result compares the iterates of $M(r)$ and $m(r)$ with those of their `relaxed' versions $\mu(r)=\varepsilon M(r)$ and $\nu(r)=m(r)/\varepsilon$, where $0<\varepsilon<1$. The analogous property for entire functions was used by Rippon and Stallard in \cite[Theorem 2.9]{rippon-stallard12}.

\begin{lem}
Let $f$ be a transcendental self-map of $\CS$, and let $\mu(r)=\varepsilon M(r)$ and $\nu(r)=m(r)/\varepsilon$, where $0<\varepsilon<1$. Then there exists $R_1(f,\varepsilon)>0$ such that, for $r\geqslant R_1(f,\varepsilon)$,
$$
\mu^n(r)\geqslant M^n(\varepsilon r)\quad \mbox{ and }\quad \nu^n(r)\leqslant m^n(\varepsilon r),\quad \mbox{ for } n>0,
$$
and, for $0<r\leqslant 1/R_1(f,\varepsilon)$,
$$
\mu^n(r)\geqslant M^n(r/\varepsilon)\quad \mbox{ and }\quad \nu^n(r)\leqslant m^n(r/\varepsilon),\quad \mbox{ for } n>0.
$$
\label{lem:mu}
\end{lem}
\begin{proof}
Let $R$ be large enough that $M(\varepsilon r)\geqslant \varepsilon r$ for all $r\geqslant R$. By property (iv) in Lemma \ref{lem:mrf}, with $k=1/\varepsilon$, there is $R'\geqslant R$ such that,
$$
\dfrac{M(r)}{M(\varepsilon r)}  \geqslant \frac{1}{\varepsilon^2},\quad \mbox{ for } r\geqslant R', 
$$
and therefore
$$
\mu(r)=\varepsilon M(r)\geqslant \frac{1}{\varepsilon}M(\varepsilon r)\geqslant r,\quad \mbox{ for } r\geqslant R'.
$$
Hence, $\mu^n(r)\geqslant M^n(\varepsilon r)$, for all $n\in \N$ and $r\geqslant R'$. If $R''>0$ is the constant required for the corresponding inequality with $m(r)$ and $\nu(r)$, then we define\linebreak $S:=\max\{R', R''\}$. If $S'>0$ is the constant such that the second pair of inequalities hold for $0<r<S'$, then we put $R_1(f,\varepsilon):=\max\{S, 1/S'\}$.
\end{proof}

Finally let us prove a property of $M(r)$ and $m(r)$ that will be used later in the construction of the annular itineraries.

\begin{lem}
Let $f$ be a transcendental self-map of $\CS$, and let $\mu(r)=\varepsilon M(r)$ and $\nu(r)=m(r)/\varepsilon$, where $0<\varepsilon<1$. Then there exists $R_2(f,\varepsilon)>0$ such that, for $r\geqslant R_2(f,\varepsilon)$, 
$$
M^{n-1}(r)<\varepsilon \mu^{n}(r),\quad \mbox{ for } n>0,
$$
and, for $0<r\leqslant 1/R_2(f,\varepsilon)$,
$$
m^{n-1}(r)> \nu^{n}(r) / \varepsilon,\quad \mbox{ for } n>0.
$$
\label{lem:annular-itineraries}
\end{lem}
\begin{proof}
Consider $\tilde{\mu}(r)=\varepsilon^2M(r)$ and let $R_1(f,\varepsilon^2)>0$ be the constant defined in Lemma \ref{lem:mu}. Then
$$
\tilde{\mu}^n(r)\geqslant M^n(\varepsilon^2r),
$$
for all $n\in\N$ and $r\geqslant R_1(f,\varepsilon)$. Now let $R>R_1(f,\varepsilon)$ be large enough that $r\leqslant M(\varepsilon^2 r)$ for all $r\geqslant R$. Then, applying $M^{n-1}$ to both sides of $r\leqslant M(\varepsilon^2 r)$, we get
$$
M^{n-1}(r)\leqslant M^n(\varepsilon^2 r)\leqslant \tilde{\mu}^{n}(r),
$$
for $r\geqslant R$. Hence,
$$
M^{n-1}(r)\leqslant \tilde{\mu}^{n}(r)=\varepsilon^2M\bigl(\tilde{\mu}^{n-1}(r)\bigr)<\varepsilon^2M\bigl(\mu^{n-1}(r)\bigr)=\varepsilon \mu^n(r),
$$
for all $n\in\N$ and $r\geqslant R$. If $R'>0$ is the constant required for the corresponding inequality with $m(r)$ and $\nu(r)$, then the required result holds with\linebreak $R_2(f,\varepsilon):=\max\{R,1/R'\}$. 
\end{proof}

\section{The escaping and fast escaping sets}

\label{sec:escaping-set}

In this section we discuss some basic properties of the escaping and fast escaping sets of transcendental self-maps of $\CS$. Recall that in Definition \ref{dfn:essential-itinerary} we defined the \textit{essential itinerary} of an escaping point $z\in I(f)$ to be the symbol sequence $e=(e_n)\in \{0,\infty\}^\mathbb N$ such that
$$
e_n=\left\{
\begin{array}{ll}
0, & \mbox{ if } |f^n(z)|\leqslant 1,\vspace{10pt}\\
\infty, & \mbox{ if } |f^n(z)|> 1,
\end{array}
\right.
$$
and $I_e(f)$ denotes the set of points whose essential itinerary is eventually a shift of $e$, that is,
$$
I_e(f):=\{z\in I(f)\ :\ \exists\ell,k\in\N,\ \forall n\geqslant 0,\ |f^{n+\ell}(z)|>1~\Leftrightarrow~ e_{n+k}=\infty\}.
$$

\begin{rmk}
\begin{enumerate}
\item[i)] Observe that we used the unit circle to define the boundaries of neighbourhoods of zero and infinity but we could have used any circle $\{z\ :\ |z|=R\}$ with $R>0$, because the orbits of escaping points are eventually as close as we want to the essential singularities.
\item[ii)] These sets $I_e(f)$ are not always disjoint. In fact, $I_{e}(f)=I_{e'}(f)$ if and only if $\sigma^m(e)=\sigma^n(e')$ for some $m,n\in\N$, where $\sigma$ denotes the Bernoulli shift map. In this case we say that $e$ is \textit{equivalent} to $e'$ and write $e\cong e'$. However, it is easy to see that there are uncountably many non-equivalent essential itineraries.
\end{enumerate}
\end{rmk}

%

Recall that in Definition \ref{dfn:fast-escaping} we said that a point $z\in\C^*$ is \emph{fast escaping} if there are $\ell,k\in\N$ such that, for all $n\in\N$,\\
\begin{tabular}{l}
\tabitem $|f^{n+\ell}(z)| \leqslant R_{n+k}$, if $e_{n+k}=0$,\\
\tabitem $|f^{n+\ell}(z)| \geqslant R_{n+k}$, if $e_{n+k}=\infty$,
\end{tabular}
\hfill\stepcounter{equation}(\theequation)\\
\addtocounter{equation}{-1}
where $(R_n)$ is a sequence of positive numbers defined using the iterates of $M(r)$ and $m(r)$ according to the essential itinerary $e\in\{0,\infty\}^\N$. We now define
$$
A_{e}^{-\ell}(f,R_0):=\{z\in\C\ :\ \exists k\in\N \mbox{ such that } \refstepcounter{equation}(\theequation) \mbox{ holds for all } n\in\N\},
$$
which is a closed set. Note that $A_{e}^{-\ell}(f,R_0)=f^{-\ell}\bigl(A_{e}^0(f,R_0)\bigr)$, and since $A_{e}^{-\ell}(f,R_0)\linebreak \subseteq A_{e}^{-(\ell+1)}(f,R_0)$, the set
$$
A_e(f):=\bigcup_{\ell\in\N} A_{e}^{-\ell}(f,R_0)
$$
is a nested union of closed sets. Finally $A(f)$ stands for all the fast escaping points regardless of their essential itinerary.


The following two lemmas show that, with the definition above, the set $A(f)$ has appropriate properties for the fast escaping set.

\begin{lem}
Let $f$ be a transcendental self-map of $\CS$, and let $e\in\{0,\infty\}^\N$. There is $R(f)>0$ large enough such that if $(R_n)$ is as in Definition \ref{dfn:fast-escaping} and $R_0>R(f)$ or $R_0<1/R(f)$, then $R_n\rightarrow \{0,+\infty\}$ as $n\rightarrow \infty$. Hence, $A_e(f)\subseteq I_e(f)$.
\label{lem:Af-well-defined}
\end{lem}
\begin{proof}
Since $M(r)$ and $1/m(r)$ grow faster than any power of $r$ (see Lemma \ref{lem:mrf} (i)),\linebreak we can take $R=R(f)>0$ large enough that $M(r)>r^2$ and $1/m(r)>r^2$ if $r>R$,\linebreak and $1/M(r)<r^2$ and $m(r)<r^2$ if $0<r< 1/R$. Therefore if $R_0\in (0,+\infty)\setminus [1/R, R]$, then $R_n\rightarrow \{0,+\infty\}$ as $n\rightarrow \infty$, as required.
\end{proof}

\begin{lem}
Let $f$ be a transcendental self-map of $\CS$. For each $e\in \{0,\infty\}^\N$, $A_e(f)$ is completely invariant and independent of $R_0$, where $R_0$ satisfies the assumptions in Definition \ref{dfn:fast-escaping}. Hence, $A(f)$ is completely invariant and independent of $R_0$.
\label{lem:Af-compl-inv-and-indep-of-R}
\end{lem}
\begin{proof}
The set $A_e(f)$ is completely invariant by constuction, because if $1-\ell<0$, then
$$
f\bigl(A_e^{-\ell}(f,R_0)\bigr)=A_{e}^{-\ell+1}(f,R_0).
$$
We give the details that $A_e(f)$ is independent of $R_0$ for the case where there exists a sequence $(n_k)$ such that $R_{n_k}\rightarrow +\infty$ as $k\rightarrow \infty$. The argument when $R_n\rightarrow 0$ as $n\rightarrow \infty$ is similar. 

Suppose that $R_0'>R_0$ and let $(R_n)$ and $(R_n')$ be the sequences given by Definition \ref{dfn:fast-escaping} starting with $R_0$ and $R_0'$ respectively. Then $R_n'\rightarrow \{0,\infty\}$ as $n\rightarrow \infty$. Since $M(r)$ and $m(r)$ are unimodal functions, for large enough values of $r$ they are respectively increasing and decreasing (and the opposite when $r$ is small enough). Therefore, $A_{e}^{-\ell}(f,R_0')\subseteq A_{e}^{-\ell}(f,R_0)$ and hence
$$
\bigcup_{\ell\in\N} A_e^{-\ell}(f,R_0')\subseteq \bigcup_{\ell\in\N} A_e^{-\ell}(f,R_0).
$$
In the other direction, we use the fact that we have assumed that there is a sequence $(n_k)$ such that $R_{n_k}\rightarrow +\infty$ as $k\rightarrow\infty$. Let $m\in \N$ be such that $R_m>R_0'$. Then
$$
\bigcup_{\ell\in\N} A_e^{-\ell}(f,R_0') \supseteq \bigcup_{\ell\in\N} A_{e}^{-\ell}(f,R_m) \supseteq \bigcup_{\ell-m\in\N} A_{\tilde{e}}^{m-\ell}(f,R_0)= \bigcup_{\ell\in\N} A_{\tilde{e}}^{-\ell}(f,R_0),
$$
where $\tilde{e}$ is the symbol sequence $e$ preceded by the string $e_1\hdots e_m$. Since $\sigma^m(\tilde{e})=e$, we have
$$
A_{\tilde{e}}^{-\ell}(f,R_0)=A_e^{-\ell}(f,R_0)
$$
and therefore $A_e(f)$ is independent of the value of $R_0$.
\end{proof}



We will continue studying the dynamical and topological properties of $A(f)$ in Section \ref{sec:fast-escaping-set}.

\section{Annular itineraries}

\label{sec:annular-itineraries}

In this section we study annular itineraries for our class of functions. By Lemma~\ref{lem:mrf}, there exist $R_+,R_->0$ respectively large and small enough such that $M^n(R_+)\rightarrow +\infty$ and $m^n(R_-)\rightarrow 0$ as $n\rightarrow \infty$. We define $A_0:=A(R_-,R_+)$ and 
$$
\begin{array}{ll}
A_n:=\{z\in\C^*\ :\ M^{n-1}(R_+)\leqslant |z|< M^n(R_+)\}, & \mbox{for } n>0,\vspace{10pt}\\
A_n:=\{z\in\C^*\ :\ m^{-n}(R_-)<|z|\leqslant m^{-n+1}(R_-)\},  & \mbox{for } n<0,
\end{array}
$$
so that $\{A_n\}$ is a partition of $\C^*$. Each point $z\in I(f)$ has an associated \textit{annular itinerary} $(s_n)\in \Z^\N$ such that $f^n(z)\in A_{s_n}$. Note that this sequence depends on the values $R_-$ and $R_+$ used to define the partition. By construction, it follows from the maximum modulus principle that
$$
\begin{array}{rl}
s_{n+1}\leqslant s_n+1, & \mbox{ if } s_n>0,\vspace{10pt}\\
s_{n+1}\geqslant s_n-1, & \mbox{ if } s_n<0.
\end{array}
$$
To create escaping orbits with certain types of annular itineraries we will use the following version of a well-known result.

\begin{lem}
Let $C_m$, $m\geqslant 0$, be compact sets in $\C^*$ and $f:\C^*\rightarrow\C^*$ be a continuous function such that 
$$
f(C_m)\supseteq C_{m+1},\quad \mbox{ for } m\geqslant 0.
$$
Then $\exists \zeta$ such that $f^m(\zeta)\in C_m$, for $m\geqslant 0$.
\label{lem:covering}
\end{lem}

In our construction, the compact sets will be compact annuli $B_n\subseteq A_n$ with some covering properties. 

\begin{thm}
Let $f$ be a transcendental self-map of $\C^*$. If $\{A_n\}$ is the set of annuli defined above, then there exists a sequence of closed annuli $B_n\subseteq A_n$, $n\in\Z$, with the following covering properties:
\begin{itemize}
\item if $n>0$, $\exists k_n\leqslant 1$ such that $f(B_n)\supseteq B_{k}$ for $k_n\leqslant k\leqslant n+1$,
\item if $n<0$, $\exists k_n\geqslant -1$ such that $f(B_n)\supseteq B_{k}$ for $n-1\leqslant k\leqslant k_n$,
\end{itemize}
and $|k_m|\geqslant |k_n|$ when $|m|>|n|$ and $m,n$ have the same sign. 
\label{thm:annular-itineraries}
\end{thm}

We compare Theorem \ref{thm:annular-itineraries} with the corresponding result for entire functions \cite[Theorem 1.1]{rippon-stallard13}. In that setting there is a subsequence $(n_j)$ such that\linebreak $f(B_{n_j})\supseteq B_k$ for $0\leqslant k\leqslant n_j+1$ with at most one exception while in our case all\linebreak $f(B_n)$ cover the other $B_k$ with $0<k\leqslant n+1$ if $n>0$. Also the proof in \cite{rippon-stallard13} is significantly more involved than ours due to the possible presence of zeros of the function and multiply-connected Fatou components, and it requires the use of several new covering lemmas. 

In order to prove Theorem \ref{thm:annular-itineraries} we use the following recent covering result due to Bergweiler, Rippon and Stallard, see \cite[Theorem 3.3]{bergweiler-rippon-stallard13}. Here $\rho_X(z_1,z_2)$ stands for the hyperbolic distance between $z_1$ and $z_2$ relative to $X$, where $X$ is a hyperbolic domain, that is, $X$ has at least two boundary points.

\begin{lem}[Bergweiler, Rippon and Stallard 2013]
There exists an absolute constant $\delta>0$ such that if $f:A(R,R')\rightarrow \mathbb C^*$ is analytic, where $R'>R$, then for all $z_1,z_2\in A(R,R')$ such that 
$$
\rho_{A(R,R')}(z_1,z_2)<\delta\quad \mbox{ and }\quad |f(z_2)|\geqslant 2|f(z_1)|,
$$
we have
$$
f\bigl(A(R,R')\bigr)\supset \overline{A}\bigl(|f(z_1)|,|f(z_2)|\bigr).
$$
\label{lem:annuli-covering}
\end{lem}

Now we prove Theorem \ref{thm:annular-itineraries}.

\begin{proof}[Proof of Theorem \ref{thm:annular-itineraries}]
If $A(\varepsilon)=A(\varepsilon, \frac{1}{\varepsilon})$, $0<\varepsilon<1$, and $C=\{z : |z|=1\}$, then the hyperbolic length of $C$ with respect to $A(\varepsilon)$ is
$$
\ell_{A(\varepsilon)}(C)=\frac{\pi^2}{\log(1/\varepsilon)}
$$
(see \cite[Example 12.1]{beardon-minda07}). Since the hyperbolic length is invariant under conformal transformations, the hyperbolic length of the circle $\{z : |z|=r\}$ with respect to $A(\varepsilon r, \frac{1}{\varepsilon}r)$ is also $-\pi^2/\log \varepsilon$. We choose $0<\varepsilon<1$ to be sufficiently small that 
$$
\frac{\pi^2}{\log(1/\varepsilon)}<\delta,
$$
where $\delta$ is the absolute constant of Lemma \ref{lem:annuli-covering}.

Let $\mu(r)=\varepsilon M(r)$ and let $R_2(f,\varepsilon)$ be the constant in Lemma \ref{lem:annular-itineraries}. Then we claim that there exists $R_0=R_0(f,\varepsilon)>R_2(f,\varepsilon)>0$ such that if $R_+\geqslant R_0$, then
\begin{equation}
M^{n-1}(R_+)<\varepsilon \mu^{n}(R_+)<\mu^n(R_+)<\frac{1}{\varepsilon}\mu^{n}(R_+)<M^n(R_+)
\label{eqn:annular-itineraries}
\end{equation}
for all $n>0$. Indeed Lemma \ref{lem:annular-itineraries} ensures that the first inequality is satisfied. The two middle inequalities are clear because $0<\varepsilon <1$, and the last one is due to the fact that $\mu(r)$ is increasing for large values of $r$ and $\mu^{n-1}(R_+)<M^{n-1}(R_+)$.

\begin{figure}[h!]
\centering
\vspace{30pt}
\def\svgwidth{.5\linewidth}
\input{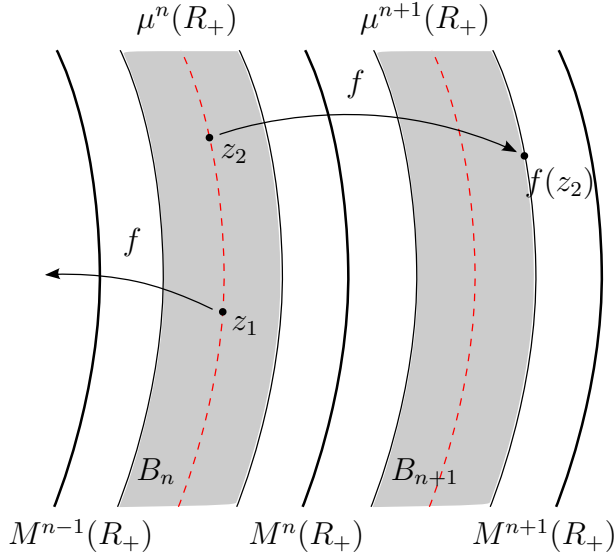}
\vspace{15pt}
\caption{Construction in the proof of Theorem \ref{thm:annular-itineraries}.}
\end{figure}

Let $B_0:=\overline{A_0}$ and, for $n>0$, define
$$
B_n:=\overline{A}\bigl(\varepsilon \mu^{n}(R_+),\ \textstyle{\frac{1}{\varepsilon}}\mu^{n}(R_+)\bigr)\subseteq A_n.
$$
Take $z_1,z_2\in B_n$ such that $|z_1|=|z_2|=\mu^n(R_+)$ and
$$
|f(z_1)|=m\bigl(\mu^n(R_+)\bigr),\quad |f(z_2)|=M\bigl(\mu^n(R_+)\bigr).
$$
This is possible by \eqref{eqn:annular-itineraries}. Our choice of $\varepsilon$ ensures that $\rho_{B_n}(z_1,z_2)<\delta$ and, for $R_+$ large enough, the condition $|f(z_2)|\geqslant 2|f(z_1)|$ is trivially satisfied. 

Finally observe that if $R_+$ is large enough we can make sure that for every $n>0$, $m\bigl(\mu^n(R_+)\bigr)<R_+$. Then Lemma \ref{lem:annuli-covering} tells us that
$$
\begin{array}{rl}
f(B_n)\hspace{-6pt}&\supseteq \overline{A}\bigl(|f(z_1)|,|f(z_2)|\bigr)=\overline{A}\Bigl(m\bigl(\mu^n(R_+)\bigr), M\bigl(\mu^n(R_+)\bigr)\Bigr)\vspace{5pt}\\ 
& \supseteq \overline{A}\left(R_+,\frac{1}{\varepsilon}\mu^{n+1}(R_+)\right)\supseteq \displaystyle{\bigcup_{j=-k_n}^{n+1} B_j}
\end{array} \vspace{-5pt}
$$
with $k_n\geqslant -1$.

The same kind of argument can be used to construct $R_->0$ and the annuli $B_n$ for $n<0$ using the iterates of the minimum modulus and of $\nu(r)=m(r)/\varepsilon$ on the circle $\{z : |z|=R_-\}$.
\end{proof}

\begin{rmk}
Note that $B_{-1},\ B_0$ and $B_1$ are the only annuli that are not compactly contained in the corresponding annulus $A_n$ because $\mu(R_+)/\varepsilon=M(R_+)$. In our construction $B_0$ is exceptional because $f(B_0)$ does not necessarily cover any $B_n$ (not even itself) whereas all the others at least cover themselves and the following one.
\end{rmk}

Theorem \ref{thm:types-annular-itineraries} describes what types of orbits can be found using the covering properties of the annuli $B_n$ that we just constructed.

\begin{proof}[Proof of Theorem \ref{thm:types-annular-itineraries}]
Using Lemma \ref{lem:covering} we deduce that if $f(B_{s_n})\supseteq B_{s_{n+1}}$, for all $n>0$, then there is a point $z_0\in B_{s_0}\subseteq A_{s_0}$ such that $f^n(z_0)\in B_{s_n}\subseteq A_{s_n}$, for all $n>0$.

We will now describe sequences that produce the various types of annular itinerary listed in Theorem \ref{thm:types-annular-itineraries}, with diagrams to illustrate each of these.
\begin{itemize}
\item The partition $\{A_n\}$ is convenient for describing fast escaping points that escape to one of the essential singularities. These correspond to annular itineraries where $s_{n+1}=s_n+1$ if $e=\overline{\infty}$ (or $s_{n+1}=s_n-1$ if $e=\overline{0}$) for $n$ arbitrarily large. 
$$
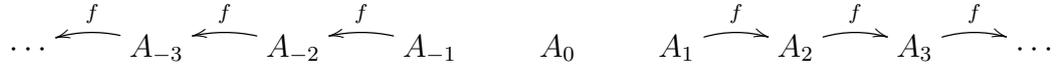
\hspace{35pt}\xymatrix{
\cdots & A_{-3} \ar@/_/[l]_f & A_{-2} \ar@/_/[l]_f & A_{-1} \ar@/_/[l]_f & A_0 & A_{1} \ar@/^/[r]^f & A_{2} \ar@/^/[r]^f & A_{3} \ar@/^/[r]^f & \cdots
}
$$

\item In order to construct a point with a periodic annular itinerary $s=\overline{s_1s_2\cdots s_n}$ we require that $s_{i+1}\in \{k_i,\hdots, s_i+1\}\setminus \{0\}$ if $s_{i+1}>0$ (or $s_{i+1}\in \{s_i-1,\hdots, k_i\}\setminus \{0\}$ otherwise) for all $1\leqslant i\leqslant n$.

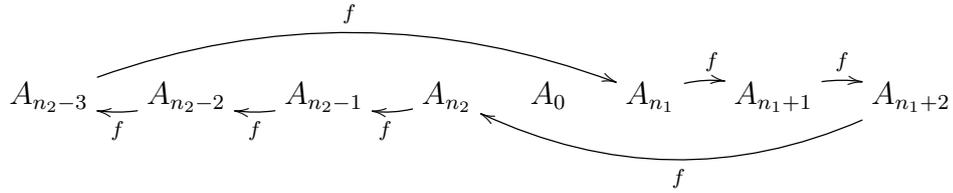
$$
\hspace{35pt}\xymatrix@=15pt{
A_{n_2-3} \ar@/^2pc/[rrrrr]^f & A_{n_2-2} \ar@/^/[l]^f  & A_{n_2-1} \ar@/^/[l]^f & A_{n_2} \ar@/^/[l]^f & A_{0}  &  A_{n_1} \ar@/^/[r]^f &  A_{n_1+1} \ar@/^/[r]^f & A_{n_1+2}  \ar@/^2pc/[llll]^f
}
$$

\item We can construct bounded itineraries whose entries are all $s_n$ or $s_{n+1}$ that are not periodic as follows. We can always choose to stay in the same annulus (every $B_n$ covers itself) or go one level up or down.
$$
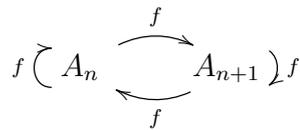
\hspace{35pt}\xymatrix{
A_n \ar@/^1pc/[r]^f \ar@(dl,ul)[]^f \ & A_{n+1} \ar@/^1pc/[l]^f \ar@(ur,dr)[]^f
}
$$
The claim that there are uncountably many such itineraries follows from the fact that at each step we always have two choices. Thus there is a bijection between this set and $2^\N \cong [0,1]$.
\vspace{5pt}

\item Unbounded non-escaping itineraries are those for which there is a sequence $(n_k)$ such that $|s_{n_k}|\rightarrow \infty$ as $k\rightarrow \infty$ and another sequence $(m_k)$ such that for all $k\in\N$, $|m_k|<R$ for some $R>0$. 
$$
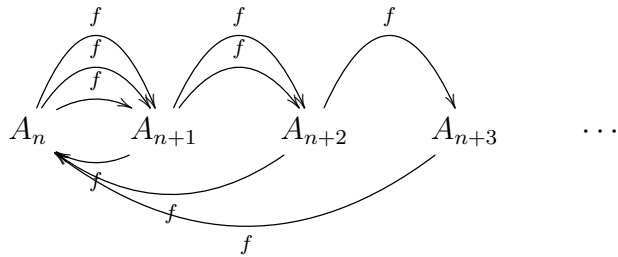
\hspace{35pt}\xymatrix{
A_n \ar@/^1pc/[r]^f \ar@/^2pc/[r]^f \ar@/^3pc/[r]^f & A_{n+1} \ar@/^2pc/[r]^f \ar@/^3pc/[r]^f \ar@/^1pc/[l]^f & A_{n+2} \ar@/^2pc/[ll]^f \ar@/^3pc/[r]^f & A_{n+3}  \ar@/^3pc/[lll]^f  & \cdots
}
$$
We are able to construct uncountably many such itineraries because we can always map to the next annulus or map back to either $B_{1}$ or $B_{-1}$.\vspace{5pt}

\item Let $e\in\{0,\infty\}^\N$ and $(r_n)$ be a sequence of positive real numbers such that
\begin{center}
$
r_n>1 \Leftrightarrow e_n=\infty, \mbox{ for all } n\in \N,\quad \mbox{ and } \quad |\log r_n|\rightarrow +\infty \mbox{ as } n\rightarrow \infty.
$
\end{center}
Then we can construct a point $z_0\in I_e(f)$ such that $|f^n(z_0)| <r_n$ if $e_n=\infty$ (resp. $|f^n(z_0)| >r_n$ if $e_n=0$) for all $n\in\N$. To do so, note that each $B_n$ covers itself, so we can choose to stay in $B_n$, $n>0$,  for as many iterates as we need so that $M\bigl(\mu^{n}(R_+)\bigr)<r_n$ and then we can choose to jump to $B_{n+1}$.
$$
\hspace{35pt}\xymatrix{
A_n \ar@/^0pc/[r]^f \ar@(ur,ul)[]^f & A_{n+1} \ar@/^0pc/[r]^f \ar@(ur,ul)[]^f & A_{n+2} \ar@/^0pc/[r]^f \ar@(ur,ul)[]^f & A_{n+3} \ar@/^0pc/[r]^f \ar@(ur,ul)[]^f & \cdots 
}
$$

\end{itemize}
This concludes the proof of Theorem \ref{thm:types-annular-itineraries}.
\end{proof}

\section{Proof of Theorem \ref{thm:fast-escaping-set-not-empty}}

\label{sec:fast-escaping-set}

In the previous section we proved the existence of points that escape as fast as possible to $\infty$ and to $0$ but we are interested in having general fast escaping points in the sense of Definition \ref{dfn:fast-escaping}, which includes points that jump infinitely many times between a neighbourhood of $\infty$ and a neighbourhood of $0$. For this we will modify the construction used to prove Theorem \ref{thm:annular-itineraries} in order to mix the iterates of $M(r)$ and $m(r)$.

\begin{proof}[Proof of Theorem \ref{thm:fast-escaping-set-not-empty}]
Let $e=(e_n)$ be an essential itinerary and let $R_0>0$ be chosen large or small enough according to whether $e_0=\infty$ or $e_0=0$. Consider the sequence given by, for $n>0$,
\begin{itemize}
\item $R_{n}=M(R_{n-1})$, if $e_{n}=\infty$,
\item $R_{n}=m(R_{n-1})$, if $e_{n}=0$.
\end{itemize}
We will show that there is a point $z$ such that, for all $n\geqslant 0$,
\begin{itemize}
\item $|f^{n}(z)| \geqslant R_n$, if $e_{n}=\infty$,
\item $|f^{n}(z)| \leqslant R_n$, if $e_{n}=0$.
\end{itemize}
Hence $z\in A_e(f)$ (note that here $\ell=k=0$). For this we will also require an auxiliary sequence $(\widetilde{R_n})$ that combines the iterates of $\mu(r)=\varepsilon M(r)$ and\linebreak $\nu(r)=m(r)/\varepsilon$, where $0<\varepsilon <1$, according to the essential itinerary $e$. Let $\widetilde{R_0}=\mu(R_0)>R_0$, if $e_0=\infty$, and $\widetilde{R_0}=\nu(R_0)<R_0$, if $e_0=0$, so that the sequence $(\widetilde{R_n})$ has a head start on $(R_n)$. For $n>0$ let
\begin{itemize}
\item $\widetilde{R_{n}}=\mu(\widetilde{R_{n-1}})$, if $e_{n}=\infty$,
\item $\widetilde{R_{n}}=\nu(\widetilde{R_{n-1}})$, if $e_{n}=0$.
\end{itemize}
Lemma \ref{lem:Af-well-defined} guarantees that, for any value of $0<\varepsilon<1$, if $R_0>R_2(f,\varepsilon)$ or $R_0<1/R_2(f,\varepsilon)$, then the sequence $(R_n)$ will accumulate to $\{0,\infty\}$ following the essential itinerary $e$. If we let $R_3(f,\varepsilon)\geqslant R_2(f,\varepsilon)$ be such that, for instance, $\varepsilon r^2>r^{3/2}$ for all $r>R_3(f,\varepsilon)$, then the sequence $(\widetilde{R_n})$ will also escape provided that $R_0>R_3(f,\varepsilon)$.

Proceeding in the same way as in the proof of Theorem \ref{thm:annular-itineraries}, we define a sequence of closed annuli $(B_n)$ such that $f(B_n)\supseteq B_{n+1}$. First, for $n>0$, we put
\begin{equation}
B_n:=\overline{A}\bigl(\varepsilon\widetilde{R_n},\ \widetilde{R_n}/\varepsilon\bigr)=\left\{
\begin{array}{l}
\overline{A}\bigl(\varepsilon^2M(\widetilde{R_{n-1}}),\ M(\widetilde{R_{n-1}})\bigr),\quad \mbox{ if } e_n=\infty,\vspace{10pt}\\
\overline{A}\bigl(m(\widetilde{R_{n-1}}),\ m(\widetilde{R_{n-1}})/\varepsilon^2\bigr),\quad \mbox{ if } e_n=0,
\end{array}
\right.
\label{eq:main-1}
\end{equation}
where $0<\varepsilon<1$ has been chosen suitably small. Next we argue as in Lemma~\ref{lem:annular-itineraries} to combine the iterates of $M(r)$ and $m(r)$, assuming that $r$ is large enough or small enough, to obtain, for $n> 0$, 
\begin{equation}
\begin{array}{cl}
R_n<\varepsilon^2M(\widetilde{R_{n-1}}), & \mbox{ if } e_n=\infty,\vspace{10pt}\\
m(\widetilde{R_{n-1}})/\varepsilon^2 < R_n, & \mbox{ if } e_n=0.
\end{array}
\label{eq:main-2}
\end{equation}
We prove \eqref{eq:main-2} by induction. The base case $n=1$ holds provided that $R_0$ (and hence $R_1$) is large enough, or small enough:
$$
\begin{array}{cl}
R_1<\varepsilon^2M(\widetilde{R_0})=\varepsilon^2M(\varepsilon R_1), & \mbox{ if } e_1=\infty \mbox{ and } e_0=\infty,\vspace{10pt}\\
R_1<\varepsilon^2M(\widetilde{R_0})=\varepsilon^2M(R_1/\varepsilon), & \mbox{ if } e_1=\infty \mbox{ and } e_0=0,\vspace{10pt}\\
m(\varepsilon R_1)/\varepsilon^2=m(\widetilde{R_{0}})/\varepsilon^2 < R_1, & \mbox{ if } e_1=0 \mbox{ and } e_0=\infty,\vspace{10pt}\\
m(R_1/\varepsilon)/\varepsilon^2=m(\widetilde{R_{0}})/\varepsilon^2 < R_1, & \mbox{ if } e_1=0 \mbox{ and } e_0=0.
\end{array}
$$
Let $e_{n}=\infty$ and $e_{n+1}=0$, and suppose $R_n<\varepsilon^2M(\widetilde{R_{n-1}})$. Then,
$$
m(\widetilde{R_{n}})=m\bigl(\varepsilon M(\widetilde{R_{n-1}})\bigr)<\varepsilon^2m\bigl(\varepsilon^2M(\widetilde{R_{n-1}})\bigr) < \varepsilon^2 m(R_n)=\varepsilon^2 R_{n+1},
$$
as required.~Note that the first inequality here is due to the fact that\linebreak $m(r)<\varepsilon^2m(\varepsilon r)$ for $0<\varepsilon<1$ and $r>0$ sufficiently large (see Lemma \ref{lem:mrf} (i)),\linebreak while in the second inequality we use the induction hypothesis. The other three possible combinations of $e_{n}$ and $e_{n+1}$ follow similarly.

Thus, we have, for $n>0$, by \eqref{eq:main-1} and \eqref{eq:main-2},
$$
\begin{array}{cl}
B_n\subseteq A\bigl(R_n,\ M(R_n)\bigr)  \subseteq \C\setminus D(0,R_n), & \mbox{ if } e_n=\infty,\vspace{10pt}\\
B_n\subseteq A\bigl(m(R_n),\ R_n\bigr)\subseteq D(0,R_n),  & \mbox{ if } e_n=0.
\end{array}
$$
Thus, taking $z_1,z_2\in B_n$ such that $|z_1|=|z_2|=\widetilde{R_n}$ and
$$
|f(z_1)|=m(\widetilde{R_{n}}),\quad |f(z_2)|=M(\widetilde{R_{n}}),
$$
and applying Lemma \ref{lem:annuli-covering} we deduce that $f(B_n)$ always covers $B_{n+1}$ and therefore, by Lemma \ref{lem:covering}, $A_e(f)\neq \emptyset$.

Finally, Baker showed that transcendental self-maps of $\C^*$ can only have one doubly-connected Fatou component, which must separate 0 from $\infty$ \cite[Theorem 1]{baker87}. Thus, $B_n\cap J(f) \neq\emptyset$ for all $n$ large enough and hence $A_e(f)\cap J(f)\neq \emptyset$.  
\end{proof}

\section{Further properties of the escaping and fast escaping sets}

\label{sec:the-rest}

Here we prove Theorems \ref{thm:boundaries} and \ref{thm:components-closure-If-unbdd} which correspond to properties (I2) and (I3) proved by Eremenko in \cite{eremenko89} for transcendental entire functions. In Theorem \ref{thm:cc-fast-escaping-set-unbounded} we show that the components of the fast escaping set are unbounded. Before this, we prove Theorem \ref{thm:fast-escaping-Fatou-components} which concerns fast escaping Fatou components and is of independent interest. 

\begin{thm}
Let $U$ be a Fatou component of $f$ such that $U\cap A_{e}^{-\ell}(f,R)\neq \emptyset$ for some $R>0$, $\ell\in\N$ and $e\in\{0,\infty\}^\N$, then $\overline{U}\subseteq A_{e}^{-\ell}(f,R)$.
\label{thm:fast-escaping-Fatou-components}
\end{thm}

We will use the following distortion lemma of Baker \cite{baker88} (see also Lemma 7 in \cite{bergweiler93}).

\begin{lem}[Baker 1988]
Let $G$ be an unbounded open set in $\C$ with at least two finite boundary points, and let $g$ be analytic in $G$. Let $D$ be a domain contained in $G$, and suppose that $g^n(D)\subseteq G$ for all $n$ and that $g^n|_D\rightarrow \infty$ as $n\rightarrow \infty$. Let $K$ be a compact subset of $D$. If $\CR\setminus G$ contains a connected set $\Gamma$ such that $\{a,\infty\}\subseteq \Gamma$ for some $a\in \C$, then there exist constants $C$ and $n_0$ such that
$$
|g^n(z_1)|\leqslant C|g^n(z_2)|
$$
for all $z_1,z_2\in K$ and $n\geqslant n_0$.
\label{lem:distortion-escaping-Fatou-comp}
\end{lem}

\begin{proof}[Proof of Theorem \ref{thm:fast-escaping-Fatou-components}]
We can assume without loss of generality that $U$ is simply-connected. Otherwise we can argue in a similar way with $f(U)$, which must be simply-connected because $f$ can only have one multiply-connected Fatou component \cite[Theorem 1]{baker87}. Let $(R_n)$ be the sequence starting with $R_0>R(f)$ (see Lemma \ref{lem:Af-well-defined}) and defined iteratively by $R_n=\widetilde{M}(R_{n-1})$ where $\widetilde{M}(r)$ is $M(r)$ or $m(r)$ according to the essential itinerary $e$. We assume that there is a subsequence $(n_j)$ such that $R_{n_j}\rightarrow +\infty$ as $j\rightarrow \infty$, the proof is similar in the other case. If $z_0\in U\cap A_e^{-\ell}(f,R_0)$, where $\ell\in\N$, then 
\begin{equation}
|f^{n_j+\ell}(z_0)|\geqslant \widetilde{M}^{n_j}(R_0)=R_{n_j},\quad \mbox{ for } j\in\N.
\label{eq:escaping1}
\end{equation}
Note that by normality the whole component $U$ is in $I(f)$ and (by Lemma \ref{lem:distortion-escaping-Fatou-comp}) all points in $U$ have the same essential itinerary $e$. 

Suppose now that there is $z_1\in U\setminus A_e^{-\ell}(f,R_0)$. This means that there is $N=n_k$ for some $k\in\N$ and $c>1$ such that
$$
R(f)<|f^{N+\ell}(z_1)|=\widetilde{M}^{N}(R_0)^{1/c}=R_{N}^{1/c}=:K
$$
and hence, by the definition of $\widetilde{M}(r)$,
\begin{equation}
|f^{n_j+\ell}(z_1)|\leqslant \widetilde{M}^{n_j-N}(R_{N}^{1/c}) ,\quad \mbox{ for } j\in\N \mbox{ such that } n_j>N.
\label{eq:escaping2}
\end{equation}
We can suppose that $K$ is larger than the constant $R=R(f)$ from Lemma \ref{lem:mrf} (iii) such that $M(r^k)\geqslant M(r)^k$ for $r\geqslant R$ and $k>1$. Then, combining equations \eqref{eq:escaping1} and \eqref{eq:escaping2}, we obtain
$$
\frac{|f^{n_j+\ell}(z_0)|}{|f^{n_j+\ell}(z_1)|}\geqslant \frac{\widetilde{M}^{n_j}(R_0)}{\widetilde{M}^{n_j-N}(R_{N}^{1/c})}=\frac{\widetilde{M}^{n_j-N}\bigl(\widetilde{M}^{N}(R_0)\bigr)}{\widetilde{M}^{n_j-N}(K)}=\frac{\widetilde{M}^{n_j-N}(K^c)}{\widetilde{M}^{n_j-N}(K)}
$$
for all $n_j>N$. This contradicts Lemma \ref{lem:distortion-escaping-Fatou-comp} because, by Lemma \ref{lem:mrf} (iii),
$$
\frac{\widetilde{M}^{n_j-N}(K^c)}{\widetilde{M}^{n_j-N}(K)}=\frac{\bigl(\widetilde{M}^{n_j-N}(K)\bigr)^c}{\widetilde{M}^{n_j-N}(K)}=\bigl(\widetilde{M}^{n_j-N}(K)\bigr)^{c-1}\rightarrow +\infty \mbox{ as } j\rightarrow \infty.
$$
Therefore $U\subseteq A_e^{-\ell}(f,R)$ and, since $A_e^{-\ell}(f,R)$ is closed, $\overline{U}\subseteq A_e^{-\ell}(f,R)$.


\end{proof}


Now we can prove Theorem \ref{thm:boundaries}, which says that, for each $e\in \{0,\infty\}^\mathbb N$,\linebreak $J(f)=\partial A_e(f)=\partial I_e(f)$ and also that $J(f)=\partial A(f)=\partial I(f)$.

\begin{proof}[Proof of Theorem \ref{thm:boundaries}]
Take $z\in J(f)$, and let $V$ be a neighbourhood of $z$. Consider $z_1\in A_e(f)\subseteq I(f)$ and let $z_2=f(z_1)\neq z_1$. Since the family of iterates of $f$ is not normal in $V$, by Montel's theorem we can find a preimage $z^*$ of $z_1$ or $z_2$ in $V$, that is, $f^k(z^*)\in\{z_1,z_2\}$ for some $k\geqslant 1$. Since $A_e(f)$ and $I(f)$ are completely invariant, $z^*\in A_e(f)\subseteq I(f)$. Thus $J(f)\subseteq \overline{A_e(f)}$. But $\interior A_e(f)\subseteq \interior I(f)\subseteq F(f)$ because periodic points are dense in $J(f)$. So $J(f)\subseteq \partial A_e(f)$.

The opposite inclusion follows from Theorem \ref{thm:fast-escaping-Fatou-components}. If there exists a point\linebreak $z\in \partial A_e(f)\cap~F(f)$, then there would be points arbitrarily close to $z$ in $A_e(f)$ but since $F(f)$ is open the whole Fatou component would be in $A_e(f)$. Hence $\partial A_e(f)\subseteq J(f)$.

The facts that $J(f)=\partial I_e(f)$ for each $e\in\{0,\infty\}^\N$ and $J(f)=\partial A(f)=\partial I(f)$ are proved similarly.
\end{proof}

Observe that $\{A_e(f)\}$ and $\{I_e(f)\}$ are both uncountable collections of disjoint sets all sharing the same boundary, which is precisely the Julia set $J(f)$. Baker, Dom\'inguez and Herring \cite{baker-dominguez-herring01} had shown previously that if $f$ is a meromorphic function with a certain set of essential singularities $E$ then the set of points escaping to one particular $e\in E$, namely $I(f,e)$, satisfies that $\partial I(f, e)=J(f)$. This implies that in our setting $\partial I_0(f)=\partial I_\infty(f)=J(f)$ which was also shown by Fang \cite{liping98}. Our result shows that this property holds for $I_e(f)$ for any essential itinerary $e\in\{0,\infty\}^\N$.

Next we prove Theorem \ref{thm:components-closure-If-unbdd}.~Recall that a set $X$ is unbounded in $\CS$ if\linebreak $\overline{X}\cap \{0,\infty\}\neq \emptyset$.

\begin{proof}[Proof of Theorem \ref{thm:components-closure-If-unbdd}]
Suppose to the contrary that $X$ is a component of $\overline{I_e(f)}$ that is bounded away from~$0$ and $\infty$. Then there is a topological annulus $A$ in the complement of $I_e(f)$ separating $X$ from both $0$ and $\infty$. Since the points in $A$ have orbits that miss $I_e(f)$, which consists of infinitely many points, $A\subseteq F(f)$ by Montel's theorem. Let $K$ be the component of $\C^*\setminus A$ containing $X$. By Theorem \ref{thm:boundaries}, $K\cap J(f)\neq \emptyset$ and hence $A$ must be contained in a multiply-connected component of $F(f)$. But Baker and Dom\'inguez showed that such components must be doubly-connected and separate $0$ from $\infty$ \cite{baker-dominguez98} which is a contradiction to the fact that $A$ is doubly-connected and separates a component of $J(f)$ from both $0$ and $\infty$. 

The last claim in the statement of the theorem follows from the fact that every connected component of $\overline{I(f)}$ contains at least one component of $\overline{I_e(f)}$ for some $e\in\{0,\infty\}^\N$, and hence it must be unbounded as well.
\end{proof}

Before proving Theorem \ref{thm:cc-fast-escaping-set-unbounded} we need the following lemma concerning preimages of unbounded closed sets under transcendental self-maps of $\C^*$.

\begin{lem}
Let $f$ be a transcendental self-map of $\CS$, and let $X\subseteq \C^*$ be an unbounded continuum. Then all the components of $f^{-1}(X)$ are unbounded.
\label{lem:unbounded}
\end{lem}
\begin{proof}
Let $W$ be a connected component of $f^{-1}(X)$. Since $f$ is continuous and $X$ is closed, $f^{-1}(X)$ is also closed. Assume, to the contrary, that $W$ is bounded. Then there exists a Jordan curve $\gamma \subseteq \CS\setminus f^{-1}(X)$ such that $W\subseteq \mbox{int}\,\gamma$. Since $f$ is an open mapping, $f(\mbox{int}\,\gamma)$ is a connected open set such that $f(\mbox{int}\,\gamma)\subseteq \C^*$ and $\partial f(\mbox{int}\,\gamma)\subseteq f(\gamma)$ which does not meet $X$. Thus $X\subseteq f(\mbox{int}\,\gamma)$ which contradicts the fact that $X$ is unbounded.
\end{proof}

Finally we prove Theorem \ref{thm:cc-fast-escaping-set-unbounded} which says that the connected components of $A(f)$ are all unbounded. Note that in particular this implies that $I(f)$ has at least one component which has $0$ in its closure and one component (possibly the same component) which has $\infty$ in its closure. 

\begin{proof}[Proof of Theorem \ref{thm:cc-fast-escaping-set-unbounded}]
Let $z_0\in A_e(f)$. For simplicity we assume that $\ell=k=0$. Fix $n\in\N$ and suppose that $e_n=\infty$, that is,
$$
|f^n(z_0)|>R_n=M(R_{n-1}).
$$ 
Consider the finite sequence of closed sets
$$
X_{n,j}:=f^{-j}\bigl(\C\setminus D(0,R_n)\bigr),\quad j=1,\hdots,n,
$$
which, by Lemma \ref{lem:unbounded}, are unbounded. One of the connected components of $X_{n,j}$ must contain the point $f^{n-j}(z_0)$; we denote this component by $L_{n,j}$.

Now there are two cases to consider: either
\begin{enumerate}
\item[(i)] $e_{n-1}=\infty$ and $|f^{n-1}(z_0)|>M(R_{n-2})=R_{n-1}$, or
\item[(ii)] $e_{n-1}=0$ and $|f^{n-1}(z_0)|<m(R_{n-2})=R_{n-1}$.
\end{enumerate}
In case (i), $L_{n,1}$ cannot contain points of modulus less than $R_{n-1}$. Otherwise if~$w$ is such that $|w|<R_{n-1}$ and $f(w)\in L_{n,0}=\C\setminus D(0,R_n)$ then we would get a contradiction with the fact that $R_n=M(R_{n-1})$ but $|f(w)|>R_n$. Similarly, in case (ii), if $|f^{n-1}(z_0)|<R_{n-1}$ we cannot have points in $L_{n,1}$ that have modulus larger than $R_{n-1}$. 

Now, iterating this procedure, for every $n\in\N$, we deduce that $L_n=L_{n,n}$ is a closed connected set which is contained in $\C\setminus D(0,R_0)$, if $e_0=\infty$, or in $D(0,R_0)$, if $e_0=0$. Observe that 
$$
L_{n+1}\subseteq L_n.
$$
Otherwise there would exist $w\in L_{n+1}$ such that $w\notin L_n$ which means that
$$
\begin{array}{cl}
|f^{n+1}(w)|>R_{n+1},& \mbox{ if } e_n=\infty,\vspace{10pt}\\
|f^{n+1}(w)|<R_{n+1},& \mbox{ if } e_n=0,
\end{array}
$$
but
$$
\begin{array}{cl}
|f^{n}(w)|<R_{n},& \mbox{ if } e_n=\infty,\vspace{10pt}\\
|f^{n}(w)|>R_{n},& \mbox{ if } e_n=0,
\end{array}
$$
which is a contradiction. Therefore $(L_n\cup \{e_0\})$ is a nested sequence of continua all containing $z_0$ and $e_0$, and hence
$$
K=\bigcap_{n\in\N} (L_n\cup \{e_0\})
$$
is also a continuum in $\CR$ which contains $z_0$ and $e_0$. Let $\Gamma$ be the connected component of $K\setminus \{e_0\}$ that contains $z_0$. Then $\Gamma$ is closed and unbounded. Here we are using the following result from continuum theory: if $E_0$ is a continuum in $\CR$, $E_1$ is a closed subset of $E_0$, and $C$ is a component of $E_0\setminus E_1$ then $\overline{C}$ meets $E_1$ \cite[pp. 84]{newman61}. Since $\Gamma\subset A_e(f)$, the theorem is proved.
\end{proof}

\bibliography{bibliography.bib}

\end{document}